\documentclass[10pt,a4paper]{amsart}
\usepackage[margin=20mm]{geometry}
\usepackage[numbers]{natbib}
\usepackage{hyperref}

\usepackage{amssymb,amsmath}
\usepackage{amsthm}
\usepackage{bbm}
\usepackage{stmaryrd}
\usepackage[usenames,dvipsnames]{xcolor}
\usepackage{color}

\input xy
\xyoption{all} 

\numberwithin{equation}{section}

\newtheorem{theorem}{Theorem}[section]

\newtheorem{proposition}[theorem]{Proposition}

\theoremstyle{definition}
\newtheorem{definition}[theorem]{Definition}
\newtheorem{remark}[theorem]{Remark}
\newtheorem{example}[theorem]{Example}

\newtheorem{construction}[theorem]{Construction}

\newenvironment{warning}[1][Warning.]{\begin{trivlist}
\item[\hskip \labelsep {\bfseries #1}]}{\end{trivlist}}

\newcommand{\rmd}{\textnormal{d}}

\DeclareMathOperator{\Vect}{Vect}

\font\black=cmbx10 \font\sblack=cmbx7 \font\ssblack=cmbx5 \font\blackital=cmmib10  \skewchar\blackital='177
\font\sblackital=cmmib7 \skewchar\sblackital='177 \font\ssblackital=cmmib5 \skewchar\ssblackital='177
\font\sanss=cmss10 \font\ssanss=cmss8 
\font\sssanss=cmss8 scaled 600 \font\blackboard=msbm10 \font\sblackboard=msbm7 \font\ssblackboard=msbm5
\font\caligr=eusm10 \font\scaligr=eusm7 \font\sscaligr=eusm5  \font\fraktur=eufm10
\font\sfraktur=eufm7 \font\ssfraktur=eufm5 
\font\bsymb=cmsy10 scaled\magstep2
\def\all#1{\setbox0=\hbox{\lower1.5pt\hbox{\bsymb
       \char"38}}\setbox1=\hbox{$_{#1}$} \box0\lower2pt\box1\;}
\def\exi#1{\setbox0=\hbox{\lower1.5pt\hbox{\bsymb \char"39}}
       \setbox1=\hbox{$_{#1}$} \box0\lower2pt\box1\;}

\def\tx#1{{\fam0\relax#1}}

\newfam\bifam
\textfont\bifam=\blackital \scriptfont\bifam=\sblackital \scriptscriptfont\bifam=\ssblackital

\newfam\blfam
\textfont\blfam=\black \scriptfont\blfam=\sblack \scriptscriptfont\blfam=\ssblack

\newfam\bbfam
\textfont\bbfam=\blackboard \scriptfont\bbfam=\sblackboard \scriptscriptfont\bbfam=\ssblackboard

\newfam\ssfam
\textfont\ssfam=\sanss \scriptfont\ssfam=\ssanss \scriptscriptfont\ssfam=\sssanss
\def\sss#1{{\fam\ssfam\relax#1}}

\newfam\clfam
\textfont\clfam=\caligr \scriptfont\clfam=\scaligr \scriptscriptfont\clfam=\sscaligr

\newfam\frfam
\textfont\frfam=\fraktur \scriptfont\frfam=\sfraktur \scriptscriptfont\frfam=\ssfraktur

\def\hpb#1{\setbox0=\hbox{${#1}$}
    \copy0 \kern-\wd0 \kern.2pt \box0}
\def\vpb#1{\setbox0=\hbox{${#1}$}
    \copy0 \kern-\wd0 \raise.08pt \box0}

\def\pmb#1{\setbox0\hbox{${#1}$} \copy0 \kern-\wd0 \kern.2pt \box0}
\def\pmbb#1{\setbox0\hbox{${#1}$} \copy0 \kern-\wd0
      \kern.2pt \copy0 \kern-\wd0 \kern.2pt \box0}
\def\pmbbb#1{\setbox0\hbox{${#1}$} \copy0 \kern-\wd0
      \kern.2pt \copy0 \kern-\wd0 \kern.2pt
    \copy0 \kern-\wd0 \kern.2pt \box0}
\def\pmxb#1{\setbox0\hbox{${#1}$} \copy0 \kern-\wd0
      \kern.2pt \copy0 \kern-\wd0 \kern.2pt
      \copy0 \kern-\wd0 \kern.2pt \copy0 \kern-\wd0 \kern.2pt \box0}
\def\pmxbb#1{\setbox0\hbox{${#1}$} \copy0 \kern-\wd0 \kern.2pt
      \copy0 \kern-\wd0 \kern.2pt
      \copy0 \kern-\wd0 \kern.2pt \copy0 \kern-\wd0 \kern.2pt
      \copy0 \kern-\wd0 \kern.2pt \box0}

\mathchardef\za="710B  
\mathchardef\zb="710C  
\mathchardef\zg="710D  
\mathchardef\zd="710E  
\mathchardef\zve="710F 
\mathchardef\zz="7110  
\mathchardef\zh="7111  
\mathchardef\zvy="7112 
\mathchardef\zi="7113  
\mathchardef\zk="7114  
\mathchardef\zl="7115  
\mathchardef\zm="7116  
\mathchardef\zn="7117  
\mathchardef\zx="7118  
\mathchardef\zp="7119  
\mathchardef\zr="711A  
\mathchardef\zs="711B  
\mathchardef\zt="711C  
\mathchardef\zu="711D  
\mathchardef\zvf="711E 
\mathchardef\zq="711F  
\mathchardef\zc="7120  
\mathchardef\zw="7121  
\mathchardef\ze="7122  
\mathchardef\zy="7123  
\mathchardef\zf="7124  
\mathchardef\zvr="7125 
\mathchardef\zvs="7126 
\mathchardef\zf="7127  
\mathchardef\zG="7000  
\mathchardef\zD="7001  
\mathchardef\zY="7002  
\mathchardef\zL="7003  
\mathchardef\zX="7004  
\mathchardef\zP="7005  
\mathchardef\zS="7006  
\mathchardef\zU="7007  
\mathchardef\zF="7008  
\mathchardef\zW="700A  
\mathchardef\zC="7009  

\newcommand{\be}{\begin{equation}}
\newcommand{\ee}{\end{equation}}

\newcommand{\bea}{\begin{eqnarray}}
\newcommand{\eea}{\end{eqnarray}}
\def\*{{\textstyle *}}
\newcommand{\R}{{\mathbb R}}

\newcommand{\Z}{{\mathbb Z}}

\newcommand{\s}{{\textstyle *}}

\def\Vect{\sss{Vect}}

\def\sT{{\sss T}}

\def\xi{\tx{i}}

\def\s*{{\scriptstyle *}}

\def\cO{\mathcal{O}}

\newcommand{\beas}{\begin{eqnarray*}}
\newcommand{\eeas}{\end{eqnarray*}}

\def\half{\frac{1}{2}}

\title{The super-Sasaki metric on the antitangent bundle }

   \author{Andrew James Bruce} 
   \address{Mathematics Research Unit, University of Luxembourg, Maison du Nombre 6, avenue de la Fonte, 
L-4364 Esch-sur-Alzette}  
   \email{andrewjamesbruce@googlemail.com}
  
\date{\today}
 
\begin{document}

\begin{abstract}
We show how to lift a Riemannian metric and almost symplectic form on a manifold  to a Riemannian structure on a canonically associated supermanifold known as the antitangent or shifted tangent bundle. We view this construction as a generalisation of Sasaki's construction of a Riemannian metric on the tangent bundle of a Riemannian manifold.\par
\end{abstract}

 \maketitle

\setcounter{tocdepth}{2}

\section{Introduction} 
In this note, we show how one can, in a canonical and well-defined way, construct an even Riemannian metric on the supermanifold $ \Pi \sT M$ provided that the  (pure even) manifold $M$ is equipped with a Riemannian (or pseudo-Riemannian) structure and an almost symplectic structure (see Construction \ref{cons:RiemStru}). This construction first appears in an earlier paper by the author  (see \cite{Bruce:2020}) as an example of a Riemannian supermanifold, albeit details are missing from that paper.   We view this construction as giving a super-version of a Sasaki metric on the antitangent rather than the tangent bundle \cite{Sasaki:1958}.  We then briefly investigate some of the properties of the constructed metric. Recall that differential forms on the manifold $M$ are functions on the supermanifold $\Pi \sT M$.  Furthermore, the associated Cartan calculus has a neat supergeometric formulation as vector fields on the said supermanifold. Thus, we examine the constructed metric on the de Rham differential, the interior derivative and the Lie derivative (see Proposition \ref{prop:Cartan}). In doing so, we recover aspects of the underlying Riemannian and almost symplectic geometry. \par 
As a concrete and possibly physically interesting example, we show that the antitangent bundle of Misner space \cite{Misner:1967} comes canonically equipped with a super-Sasaki metric.  We explicitly construct the super-Sasaki metric for this example. As far as we know, this metric is completely new. For details see Example \ref{exp:Misner}. \par 
It is well-known that any smooth manifold can be equipped with a Riemannian metric, i.e., there are no topological obstructions to the existence of a metric. The same is not true for pseudo-Riemannian structures; there are certain topological obstructions. Similarly, not every smooth manifold can be equipped with an almost symplectic form, i.e., a non-degenerate two form that is not necessarily closed.  For one, the non-degeneracy forces the dimension of the manifold to be even.  A necessary and sufficient condition for the existence of an almost symplectic structure on a manifold is the possibility of reducing the structure group of its tangent bundle to a unitary group.  This forces the odd-dimensional Stiefel--Whitney classes of the manifold to vanish, see Libermann \cite{Libermann:1955} for more details.  Thus it is far from true that our construction can be applied to any manifold. We require the assumption that we have at least an almost symplectic structure. A natural class of almost symplectic manifolds are, or course, symplectic manifolds. The integrability condition, i.e., being closed, plays no r\^{o}le in our constructions. For examples of almost symplectic manifolds see Vaisman \cite{Vaisman:2013}. Examples of manifolds that have both a Riemannian and almost symplectic structures include almost Hermitian manifolds and almost K\"{a}hler  manifolds.  In these examples there is an additional compatible almost complex structure. In fact, it is well-known that any almost symplectic manifold equipped with a Riemannian metric admits a compatible  canonical almost complex structure, such manifolds are thus canonically almost Hermitian manifolds (for details see, for example \cite[Chapter 4]{McDuff:2017}). However, almost complex structures themselves  play no direct r\^{o}le in our construction of Sasaki-like metrics on the antitangent bundle. We will briefly comment further on this in the closing section of this note.\par 
It remains to properly explore the interrelations between the geometries of the initial almost Hermitian manifold and the lifted Riemannian geometry on the antitangent bundle.  It is an open question as to what results from the classical study of Sasaki metrics directly generalise to  super-Sasaki metrics. We have only scratched the surface here.\par 

We remark that the classical Sasaki metric makes a natural appearance in geometric mechanics. Specifically, one can consider the configuration space of a system equipped with the Jacobi metric and then, in turn, the tangent bundle of the configuration space equipped with the associated Sasaki metric. Trajectories can be viewed as geodesics on the configuration space or its tangent bundle. For more details the reader can consult \cite[Chapter 3]{Pettini:2007}. Thus, we speculate that the super-Sasaki metric may find applications in geometric approaches to supermechanics along the lines of \cite{Bruce:2017}.

\medskip

\section{Riemannian supermanifolds}\label{sec:RiemSup}
 We will assume that the reader is familiar with the basic theory of supermanifolds. In this note we will understand a \emph{supermanifold} $M := (|M|, \:  \cO_{M})$ of dimension $n|m$  to be a supermanifold as defined by Berezin \& Leites  \cite{Berezin:1976}, i.e., as a locally superringed space that is locally isomorphic to $\mathbb{R}^{n|m} := \big (\R^{n}, C^{\infty}(\R^{n})\otimes \Lambda(\zx^{1}, \cdots \zx^{m}) \big)$. Here, $\Lambda(\zx^{1}, \cdots \zx^{m})$ is the Grassmann algebra (over $\R$) with $m$ generators. Given any point on $|M|$ we can always find a `small enough' open neighbourhood $|U|\subseteq |M|$ such that we can  employ local coordinates $x^{a} := (x^{\mu} , \zx^{i})$ on $M$. The (global) sections of the structure sheaf we will refer to as \emph{functions}. The  supercommutative algebra of functions we will denote as $C^{\infty}(M)$. The underlying smooth manifold $|M|$ we refer to as the \emph{reduced manifold}.  We will make heavy use of local coordinates on supermanifolds and employ the standard abuses of notation when it comes to morphisms and changes of coordinates.  The Grassmann parity of an object $A$ will be denoted by `tilde', i.e., $\widetilde{A} \in \Z_{2}$.  By `even' or `odd' we will be referring to the  Grassmann parity of an object. \par
The \emph{tangent sheaf} $\mathcal{T}M$ of a supermanifold $M$ is  defined as the sheaf of derivations of sections of the structure sheaf. Naturally, this is  a sheaf of locally free $\cO_{M}$-modules. Global sections of the tangent sheaf are \emph{vector fields}. We denote the $\cO_{M}(|M|)$-module of vector fields as $\Vect(M)$. The total space of the tangent sheaf $\sT M$ we will refer to  as the \emph{tangent bundle}.  By shifting the parity of the fibre coordinates one obtains the \emph{antitangent bundle} $\Pi \sT M$.  We will also encounter \emph{double supervector bundles}, namely $\sT(\Pi \sT M)$. The notion original of  double vector bundles from the point of view of category theory is due to Pradines \cite{Pradines:1974}. An approach closer to classical differential geometry was formulated by   Konieczna \& Urba\'{n}ski \cite{Konieczna:1999}.\par 
We proceed to present the bare minimum of the theory of Riemannian supermanifolds as needed in this note. Relatively recent papers on Riemannian supergeometry include \cite{Galaev:2012,Garnier:2014,Garnier:2012,Goertsches:2008,Groeger:2014,Klaus:2019}. We suggest  Carmeli, Caston \& Fioresi \cite{Carmeli:2011},  Manin \cite{Manin:1997}  and Varadarajan \cite{Varadarajan:2004} as general references for the theory of supermanifolds. The encyclopedia edited by Duplij, Siegel \& Bagger  \cite{Duplij:2004} is also very helpful, as is the review paper by Leites \cite{Leites:1980}.   
\begin{definition}\label{def:RiemMet}
A  \emph{Riemannian metric} on a supermanifold $M$ is grading preserving, symmetric, non-degenerate, $\cO_M$-linear morphisms of sheaves
$$\mathcal{T}M \otimes_{\cO_M} \mathcal{T}M \longrightarrow \cO_M.$$
A supermanifold equipped with a Riemannian metric is referred to as a \emph{Riemannian supermanifold}
\end{definition}
\noindent A Riemannian metric has the following properties:
\begin{enumerate}
\setlength\itemsep{0.5em}
\item$\widetilde{\langle X| Y \rangle_g} = \widetilde{X} + \widetilde{Y}$,
\item $\langle X| Y \rangle_g = (-1)^{\widetilde{X} \, \widetilde{Y}} \langle Y| X \rangle_g$,
\item If $\langle X| Y \rangle_g =0$  for all $Y \in \Vect(M)$, then $X =0$,
\item $\langle f X + Y| Z \rangle_g = f \langle X|Z \rangle_g + \langle Y| Z\rangle_g$, 
\end{enumerate} 
for arbitrary  (homogeneous) $X,Y,Z \in \Vect(M)$ and $f \in C^\infty(M)$.  \par  
Furthermore, it can be shown that a Riemannian metric on  a supermanifold $M$  induces a pseudo-Riemannian metric on the reduced manifold $|M|$. 
\begin{remark}
Odd Riemannian structures can also be defined. There are no changes to  Definition \ref{def:RiemMet}, except that  the parity is shifted, i.e., $\widetilde{\langle X| Y \rangle} = \widetilde{X} + \widetilde{Y} +1$. The non-degeneracy condition  forces the number of even and odd dimensions to be equal.  We will not encounter odd Riemannian  structures in this paper. 
\end{remark}
A Riemannian metric is completely specified by an even degree two function $g \in C^\infty(\sT M)$, i.e., a Grassmann degree zero rank $2$ symmetric covariant tensor. It is standard to refer $g$ as the metric tensor, or via minor abuse of language, as the Riemannian metric.   In local coordinates, we write
$$g(x, \dot{x}) = \dot{x}^a \dot{x}^b \,g_{ba}(x).$$  
Under changes of local coordinates $x^a \mapsto x^{a'}(x)$ the components of the metric transform as 
$$g_{b'a'}(x') = (-1)^{\widetilde{a}' \, \widetilde{b}} \left(\frac{\partial x^b}{\partial x^{b'}}\right) \left(\frac{\partial x^a}{\partial x^{a'}}\right)~  g_{ab}\,,$$
where we have used the symmetry $g_{ab} = (-1)^{\widetilde{a} \, \widetilde{b}} \, g_{ba}$.\par 
We denote the vertical lift of a vector field as $\iota_X$, which in local coordinates is given by
$$X = X^a(x)\frac{\partial}{\partial x^a} \rightsquigarrow \iota_X :=  X^{a}(x)\frac{\partial}{\partial \dot{x}^a} \in \Vect(\sT M)\,.$$
It is easy to   observe that
$$\langle X| Y \rangle_g = \frac{1}{2} \iota_X \iota_Y g\,,$$ 
which leads us to the local expression 
$$  \langle X| Y \rangle_g =  (-1)^{\widetilde{Y}\, \widetilde{a}} ~ X^a(x)Y^b(x) g_{ba}(x).$$
A direct computation will show that the above local expression for the metric pairing is invariant under changes of coordinates, we leave this as an exercise for the reader. \par 
The non-degeneracy condition of the Riemannian metric means that the dimension of the supermanifold $M$  must be $n | 2\,p$, that is, we require the number of odd ``directions'' to be even.  Just about all the standard constructions of classical Riemannian geometry generalise to Riemannian supermanifolds. For example, the fundamental theorem remains true (see for example \cite{Monterde:1996}). 
\begin{remark}
DeWitt \cite[Section 2.8]{DeWitt:1992} presents Riemannian geometry on DeWitt--Rogers $H^\infty$-supermanifolds. Although some care is needed in translating between supermanifolds (as locally ringed spaces) and DeWitt--Rogers supermanifolds, most of the expressions given by DeWitt on Riemannian structures remain valid in the locally ringed space approach to Riemannian supergeometry.
\end{remark}

\section{Construction of the Super-Sasaki metric}
\begin{warning}
From this point on $M$ will be a purely even manifold.
\end{warning}
Let $M$ be an almost symplectic manifold, i.e., a manifold equipped with a non-degenerate two-form $\omega$, we do not assume that this two-form is closed. The non-degeneracy means that the dimension of $M$ must be even. Furthermore, let us assume that $M$ is equipped with a Riemannian metric, which we will denote as $h$. We do \emph{not} require any kind of compatibility condition between the almost symplectic structure $\omega$ and the Riemannian structure $h$. From this initial data, we proceed to build a Riemannian metric (see Definition \ref{def:RiemMet}) on the supermanifold $\Pi \sT M$ in the following way.
\begin{construction}\label{cons:RiemStru}
Let $(M, h, \omega)$ be a smooth manifold equipped with a Riemannian metric and an almost symplectic structure.  The double supervector bundle $\sT(\Pi \sT M)$  we equip with natural (local) coordinates  $(x^a, \rmd x^b, \dot{x}^c , \rmd \dot{x}^d)$. Admissible changes of coordinates are of the form
\begin{align*}
&x^{a'} = x^{a'}(x), & \rmd x^{b'} = \rmd x^a \frac{\partial x^{b'}}{\partial x^a},\\
&\dot{x}^{c'} = \dot{x}^b \frac{\partial x^{c'}}{\partial x^b},& \rmd\dot{x}^{d'} = \rmd\dot{x}^{c}  \frac{\partial x^{b'}}{\partial x^c} + \dot{x}^{b} \rmd{x}^c \frac{\partial^2 x^{d'}}{ \partial x^c \partial x^b}.
\end{align*}
The Levi-Civita connection $\nabla$ associated with the Riemannian metric induces a splitting
$$\sT (\Pi \sT M) \stackrel{\phi_h}{\xrightarrow{\hspace*{25pt}}} \Pi \sT M \times_M \sT M \times_M \Pi \sT M\,,$$
which we write in natural coordinates as
$$\phi^*_h \zx^a = \rmd\dot{x}^a + \rmd{x}^b \dot{x}^c \Gamma^a_{cb}(x) =: \nabla \dot{x}^a.$$
In the above, $\zx^a$ are the (fibre) coordinates on the last factor of the decomposed double supervector bundle.  Note that while the target supermanifold of the splitting is canonical, the splitting itself is not canonical and requires an affine connection (see Remark \ref{rem:OtherConnections}).  The splitting $\phi_h$ acts as the identity on the remaining coordinates, i.e., we choose fibre coordinates $\rmd x$ and $\dot{x}$ on the first and second factors, respectively.\par 
On the decomposed  double supervector bundle $\Pi \sT M \times_M \sT M \times_M \Pi \sT M$ we can take the sum of the Riemannian metric and the almost symplectic structure. This sum is well-defined due to the linear nature of the admissible changes of coordinates on the decomposed  double supervector bundle. In natural local coordinates we have
$$G :=  h + \omega =  \dot{x}^a \dot{x}^b g_{ba}(x) + \zx^a \zx^b \omega_{ba}(x).$$
The Riemannian metric on $\Pi \sT M$ is then defined as the pull-back of $G$ by the splitting, i.e.,
\begin{equation}\label{eqn:LocMet}
g = \phi^*_h G  = \dot{x}^a \dot{x}^b g_{ba}(x) + \nabla \dot{x}^a \nabla \dot{x}^b \omega_{ba}(x).
\end{equation}
\end{construction} 
\smallskip 
\begin{remark} \label{rem:OtherConnections}
In the above construction, the Levi-Civita connection could be replaced by an arbitrary affine connection. One could, for example, use the symmetric connection associated with the almost symplectic structure. However, it is well-known that such connections are not unique. Thus, to have a canonical construction we take the Levi-Civita connection. We remark that the association of splittings of $\sT(\sT M)$ and affine connections was first presented by Dodson \& Radivoievici \cite{Dodson:1982}.
\end{remark}
Direct calculation shows that this  metric tensor on $\Pi \sT M$ can be written in natural local coordinates as
$$ g  = \dot{x}^a \dot{x}^b g_{ba}(x) + \rmd\dot{x}^a \rmd\dot{x}^b \omega_{ba}(x)
  + 2 \, \rmd\dot{x}^a \dot{x}^b\big(\rmd x^c \Gamma^d_{cb}(x)\,\omega_{da}(x)  \big) - \dot{x}^a \dot{x}^b \big( \rmd x^c \rmd x^d \Gamma^e_{da}(x) \Gamma^f_{bc}(x) \, \omega_{fe}(x) \big).
$$
By (locally) decomposing any vector field as $X = X^a(x, \rmd x)\frac{\partial}{\partial x^a} +  \bar{X}^a(x, \rmd x)\frac{\partial}{\partial \rmd x^a} \in \Vect(\Pi \sT M)$ (and similar for $Y$), we arrive at the local expression for the Riemannian metric
\begin{align}\label{eqn:LocPar}
\langle  X | Y\rangle_g & = X^a Y^b g_{ba} - X^a Y^b \big( \rmd x^c \rmd x^d \Gamma^e_{da} \Gamma^f_{bc} \, \omega_{fe} \big)\\
\nonumber  & + \big((-1)^{\widetilde{Y}}\, \bar{X}^a Y^b + (-1)^{\widetilde{X}(\widetilde{Y} +1)}\, \bar{Y}^a X^b \big) \rmd x^c \Gamma^d_{cb}\,\omega_{da}\\
 \nonumber &+ (-1)^{\widetilde{Y}} \, \bar{X}^a \bar{Y}^b \omega_{ba}.
\end{align} 
\noindent \textbf{Observations:}
\renewcommand{\theenumi}{\roman{enumi}}%
\begin{enumerate}
\item Comparing \eqref{eqn:LocPar} with Definition \ref{def:RiemMet}, we see that the parity is correct. The symmetry is similarly clear.  The non-degeneracy is follows  as both $h$ and $\omega$ are themselves non-degenerate. Linearity is also clear. Thus we do obtain a genuine Riemannian metric on $\Pi \sT M$.
\item The Riemannian metric $\langle - | -\rangle_g$ is inhomogeneous  with respect to the differential form degree, i.e., the natural $\mathbb{N}$-weight associated with the vector bundle structure of $ \Pi\sT M$.
\item The construction carries over with no modification to $h$ being a pseudo-Riemannian structure, for example we can consider Lorentzian metrics.
\end{enumerate}
\begin{definition}\label{def:SupSasMet}
Let $(M, h, \omega)$ be a smooth manifold  equipped with a Riemannian metric and an almost symplectic structure. The \emph{super-Sasaki metric} on $\Pi \sT M$  is  the Riemannian metric  given by Construction \ref{cons:RiemStru}.
\end{definition}
The local form of the Riemannian metric \eqref{eqn:LocMet} is almost identical to that of the Sasaki metric on the tangent bundle of a Riemannian manifold. In natural coordinates  $(x, \delta x, \dot{x}, \delta \dot{x})$ on $\sT (\sT M)$ the Sasaki metric, see \cite[(3.2)]{Sasaki:1958}, is 
$$g_S = \dot{x}^a \dot{x}^b \, g_{ba}(x) + D\dot{x}^a D\dot{x}^b g_{ba}(x)\,,$$
where $D\dot{x}^a :=  \delta \dot{x}^a + \delta x^b \dot{x}^c \Gamma^a_{cb}$. All the coordinates here are commuting, i.e., they are all  even coordinates. We have anticommuting coordinates on $\sT (\Pi \sT M)$ and so we require an antisymmetric structure on $M$ together with a Riemannian metric in order to generalise Sasaki's construction. This justifies our nomenclature in Definition \ref{def:SupSasMet}. An analogous construction in the context of symplectic geometry was given by Abdelhak, Boucetta \& Ikemakhen \cite{Abdelhak:2013}.
\begin{example}[Cartesian Space]
Consider $M= \R^2$, which we equip with standard Cartesian coordinates $(x,y)$. The Euclidean metric is $h = \dot{x}^2 + \dot{y}^2$, and in Cartesian coordinates the Christoffel symbols of the Levi-Civita connection are zero.  The canonical symplectic structure on $\R^2$ is, in the classical framework, $\omega = \rmd x \wedge \rmd y$. Note that
$$\sT (\Pi \sT \R^2) \cong \Pi \sT \R^2 \times_{\R^2} \sT \R^2 \times_{\R^2}  \Pi \sT \R^2 \cong \R^{4|4}. $$
We will employ natural (global) coordinates $(x,y, \zx, \eta, \dot{x}, \dot{y}, \dot{\zx} , \dot{\eta})$ on $\sT (\Pi \sT \R^2)$. Then the associated super-Sasaki metric is given by
$$g = \dot{x}^2 + \dot{y}^2 + 2\, \dot{\zx}\,\dot{\eta}\,, $$
which is the standard Euclidean  metric on the supermanifold $\R^{2|2}$. This example generalises verbatim to $\R^{2p| 2p}$ by considering the Euclidean metric and canonical symplectic structure on $\R^{2p}$. The super-Sasaki metric in this case is given by
$$g =  (\dot{x}^1)^2 + (\dot{x}^2)^2 + \cdots + (\dot{x}^{2p})^2 + 2 \, \dot{\zx}^1\,\dot{\eta}^1 + 2 \, \dot{\zx}^2\,\dot{\eta}^2 + \cdots + 2 \, \dot{\zx}^p\,\dot{\eta}^p\,.$$
This is again the standard super-Euclidean metric.
\end{example}
\begin{example}[Misner Space \cite{Misner:1967}]\label{exp:Misner}
Consider the infinite cylinder $M = \R \times S$, which we equip with coordinates $(t, \varphi)$, with $t \in \R$ and $\varphi \in [0,2 \pi)$. We give $M$ the structure of a pseudo-Riemannian  manifold by employing the Misner metric
$$h = 2 \, \dot{t} \, \dot{\varphi} + t \,(\dot{\varphi})^2\,.$$
We further equip Misner space  with the standard symplectic form, which in classical terms is $\omega = \rmd t \wedge \rmd \varphi$.  A standard calculation shows that the only non-zero Christoffel symbols (in these coordinates) are $\Gamma^t_{\varphi t} = \Gamma^t_{t \varphi } = \half$. Then in naturally induced coordinates 
\begin{align*}
\nabla \dot{t} = \rmd \dot{t}  + \half(\rmd t \, \dot{\varphi} + \rmd \varphi \, \dot{t}), && \nabla \dot{\varphi} = \rmd \dot{\varphi}  + \half(\rmd t \, \dot{\varphi} + \rmd \varphi \, \dot{t})\,.
\end{align*}
Following the construction of the super-Sasaki metric we see that
$$g = 2 \, \dot{t} \, \dot{\varphi} + t \,(\dot{\varphi})^2  + 2\,  \nabla \dot{t} \nabla \dot{\varphi}\,.$$
After a little algebra, the super-Sasaki metric on $\Pi \sT \R \times \Pi \sT S$ associated with the  Misner metric can be written as
$$g =  2 \, \dot{t} \, \dot{\varphi} + t \,(\dot{\varphi})^2  + 2 \rmd \dot{t}\, \rmd\dot{\varphi}  + (\rmd \dot{t}- \rmd\dot{\varphi})(\dot{\varphi} \, \rmd t + \dot{t}\, \rmd \varphi).$$
We remark that Misner space is used as a toy space-time as it shows many of the features of Taub-NUT space-time and other interesting space-times. In particular, Misner space has closed timelike curves and has served as a test-bed for exploring various conjectures in general relativity and semi-classical gravity (see for example Thorne \cite{Thorne:1993} and references therein).  This super-generalisation of Misner space deserves further study.
\end{example}

\begin{theorem}
Construction \ref{cons:RiemStru} is natural in the sense that if a diffeomorphism 
$$(M , h_M, \omega_M) \stackrel{\psi}{\longrightarrow} (N, h_N, \omega_N)$$
is an isometry and a symplectomorphism, then the induced diffeomorphism of supermanifolds
$$\Psi : \sT(\Pi \sT M) \longrightarrow  \sT(\Pi \sT N)$$
is an isometry, i.e., $\Psi^* g_N = g_M$.
\end{theorem}
\begin{proof}
We will prove the statement using local coordinates. Let us equip $M$ with local coordinates $x^a$ and $N$ with local coordinates $y^\alpha$.  A  diffeomorphism (or even just as smooth map) $\psi:  M \longrightarrow N$ is then locally given by $\psi^*y^\alpha = y^\alpha(x)$ (using the standard abuses of notation). The induced map $\Psi : \sT (\Pi \sT M) \longrightarrow  \sT (\Pi \sT N)$ comes from the application of the antitangent and tangent functors. Explicitly in natural coordinates, we have 
\begin{align*}
&y^{\alpha} = y^{\alpha}(x), & \rmd y^{\beta} = \rmd x^a \frac{\partial y^{\beta}(x)}{\partial x^a},\\
&\dot{y}^{\gamma} = \dot{x}^b \frac{\partial y^{\gamma}(x)}{\partial x^b},& \rmd\dot{y}^{\delta} = \rmd\dot{x}^{c}  \frac{\partial y^{\delta}(x)}{\partial x^c} + \dot{x}^{b} \rmd{x}^c \frac{\partial^2 y^{\delta}(x)}{ \partial x^c \partial x^b}.
\end{align*}
Note that by construction (using the transformation rules of Christoffel symbols and that we restrict attention to diffeomorphisms) that 
$$\Psi^*(\nabla \dot{y}^\alpha) = \nabla \dot{x}^a  \frac{\partial y^{\alpha}(x)}{\partial x^a}\,.$$
Then we insist on the following
$$\Psi^* g_N = \dot{x}^a \dot{x}^b  \frac{\partial y^\alpha}{\partial x^a} \frac{\partial y^\beta}{\partial x^b} \, g_{\beta \alpha}(y(x)) + \nabla\dot{x}^a \nabla \dot{x}^b  \frac{\partial y^\alpha}{\partial x^a} \frac{\partial y^\beta}{\partial x^b} \, \omega_{\beta \alpha}(y(x)) = g_M.$$
This means that we require 
\begin{align*}
g_{ba}(x) =  \frac{\partial y^\alpha}{\partial x^a} \frac{\partial y^\beta}{\partial x^b} \, g_{\beta \alpha}(y(x)), &&\textnormal{and}&& \omega_{ba}(x) = \frac{\partial y^\alpha}{\partial x^a} \frac{\partial y^\beta}{\partial x^b} \, \omega_{\beta \alpha}(y(x))\,,
\end{align*}
which is exactly the (local) statement that $\psi : M \longrightarrow N$ is an isometry and a symplectomorphism. 
\end{proof}
As is well-known, differential forms on $M$ are functions on the supermanifold $ \Pi \sT M$. Furthermore, the de Rham differential, the interior derivative and the Lie derivative can all be understood as vector fields on $\Pi \sT  M$. If $X \in \Vect(M)$, then (locally)
\begin{align*}
\rmd = \rmd x^a \frac{\partial}{\partial x^a}, && i_X = X^a\frac{\partial}{\rmd x^a}, && L_{X} = X^a\frac{\partial}{\partial x^a} + \rmd x^b \frac{\partial X^a}{\partial x^b} \frac{\partial}{\partial \rmd x^a}.
\end{align*}
The standard relations between these operators can easily be derived using this understanding in terms of vector fields on a supermanifold.
\begin{proposition}\label{prop:Cartan}
Let $(M, h, \omega)$ be a smooth manifold  equipped with a Riemannian metric and an almost symplectic structure, and let $\langle - | -\rangle_g$  be the associated super-Sasaki metric  on $\Pi \sT M$ (see Definition \ref{def:SupSasMet}). Let $X$ and $Y \in \Vect(M)$ be arbitrary vector fields on $M$.  Then we have the following identities:

\medskip 
\begin{enumerate}
 \setlength\itemsep{0.5em}
\item  $\langle i_X | i_Y  \rangle_g = \omega(X,Y) \in C^\infty(M)$,
\item $\langle i_X | \rmd  \rangle_g =0$,
\item $\langle \rmd | \rmd  \rangle_g=0$,
\item $ \langle L_X | \rmd  \rangle_g = X^\flat \in \Omega^1(M)$,
\item $\langle L_X | i_Y  \rangle_g = \omega(\nabla X, Y) \in \Omega^1(M)$,
\item  $\langle L_X | L_Y  \rangle_g  = \langle X| Y \rangle_h + \omega(\nabla X, \nabla Y) \in C^\infty(M) \oplus \Omega^2(M)$,
\end{enumerate}

\medskip

\noindent where  $\flat : \Vect(M) \rightarrow \Omega^1(M)$   is the standard musical isomorphism on the  Riemannian manifold $(M, h)$ and $\nabla$ is the Levi-Civita connection associated  with $h$.
\end{proposition} 
\begin{proof}
We will prove the above proposition via direct computation using local coordinates and the local expression for Riemannian metric on $\Pi \sT M$, see \eqref{eqn:LocPar}. One needs to insert the components of the vector fields into the local expression and rearrange the expressions when needed. 
\smallskip
\begin{enumerate}
 \setlength\itemsep{0.7em}
\item  $\langle i_X | i_Y  \rangle_g =  - X^a Y^b \omega_{ba}$.
\item $\langle i_X | \rmd  \rangle_g = X^a \, \rmd x^b \rmd x^c \,\Gamma^d_{cb}\omega_{da}    =0$, where we have used the fact that the Christoffel symbols are symmetric in lower indices while $\rmd x^b \rmd x^c $ is antisymmetric. 
\item $\langle \rmd | \rmd  \rangle_g = -  \langle \rmd | \rmd  \rangle_g  =0$ from the symmetry properties of a Riemannian metric on a supermanifold.
\item $ \langle L_X | \rmd  \rangle_g = \rmd x^b g_{ba}X^a$.
\item $\langle L_X | i_Y  \rangle_g = Y^a X^b \rmd x^c \Gamma_{cb}^d \omega_{da} - \rmd x^c \frac{\partial X^a}{\partial x^c} Y^b\omega_{ba} = - \rmd x^c \big(\frac{\partial X^a}{\partial x^c} + X^d \Gamma^a_{dc} \big)Y^b \omega_{ba}$.
\item  $\langle L_X | L_Y  \rangle_g  =  X^a Y^b g_{ba} - X^a Y^b \big( \rmd x^c \rmd x^d \Gamma^e_{da} \Gamma^f_{bc} \, \omega_{fe} \big) + (\rmd X^a Y^b + \rmd Y^a X^b) \rmd x^c \Gamma^d_{cb}\omega _{da} + \rmd X^a \rmd Y^b \omega_{ba}$. Now we collect terms related to $\omega$ to give $\langle L_X | L_Y  \rangle_g  = X^a Y^b g_{ba} + \rmd x^c \left( \frac{\partial X^a}{\partial x^c } + X^d \Gamma^a_{dc}\right)  \rmd x^e \left( \frac{\partial Y^b}{\partial x^e } + Y^f \Gamma^b_{fe}\right)\omega_{ba}$.
\end{enumerate}
\end{proof}
\noindent \textbf{Observations:} As a supermanifold, we have the canonical projection $\epsilon : C^\infty(\Pi \sT M) \rightarrow C^\infty(M)$. In classical language, this is just projecting inhomogeneous differential forms to their degree zero part.  We see that
$$\epsilon(\langle L_X | L_Y\rangle_g ) = \langle X | Y\rangle_h,$$
and so we can recover the initial Riemannian metric on $M$. Similarly, though this does not require the projection, we recover the almost symplectic form from $\langle i_X | i_Y\rangle_g = \omega(X,Y)$.

\section{Concluding remarks}
 It is well-known that any almost symplectic manifold equipped with a Riemannian metric $(M, h , \omega)$ admits a canonical compatible almost complex structure $J$ defined via $\omega(X,Y) = \langle J X, Y \rangle_h$.  In local coordinates this amounts to $J_a^{\:\: b} = \omega_{ac}g^{cb}$. In fact, given any pair from the compatible triple $(h, \omega, J)$, one can canonically construct the remaining structure.  Thus, we could have started with, a Riemannian or almost symplectic manifold equipped with an almost complex structure. However, as it is the metric and and the almost symplectic structure that are used in the constrictions we chose to start from there.  However one views it, we are dealing with almost Hermitian manifolds, although the almost complex structure is not explicitly needed.  None-the-less,  the main result of this note can be paraphrased as follows:

\medskip 
 
 \noindent  \emph{If $M$ is an almost Hermitian manifold, then the supermanifold $\Pi \sT M$ is canonically a Riemannian supermanifold}.
 
 \medskip 

\section*{Acknowledgements}
The author cordially thanks Steven Duplij for his helpful comments earlier drafts of this work.

\end{document}